\documentclass[10pt,reqno]{amsart}
\usepackage{amsmath}
\usepackage{bm}
\usepackage{amssymb}
\usepackage{dsfont}
\usepackage[dvips,draft,final]{graphics}
\usepackage{amssymb,amsmath}
\usepackage[T1]{fontenc}
\usepackage{fancyhdr}
\usepackage{url}
\usepackage[colorlinks,citecolor=red,pagebackref,hypertexnames=false,breaklinks]{hyperref}
\usepackage{color}
\usepackage{graphicx}

\def\squarebox#1{\hbox to #1{\hfill\vbox to #1{\vfill}}}

\newtheorem{Thm}{Theorem}[section]

\newtheorem{lem}{Lemma}[section]
\newtheorem{remark}{Remark}[section]

\numberwithin{equation}{section}

\newcommand{\bel}{\begin{equation} \label}
\newcommand{\ee}{\end{equation}}

\newcommand{\re}{\mathfrak R}

\newcommand{\R}{\mathbb{R}}

\def\epsilon{\varepsilon}
\def\phi {\varphi}

\newtheorem{prop}{Proposition}[section]

\providecommand{\abs}[1]{\left\lvert#1\right\rvert}
\providecommand{\norm}[1]{\left\lVert#1\right\rVert}
\numberwithin{equation}{section}

\renewcommand{\leq}{\leqslant}
\renewcommand{\geq}{\geqslant}
\providecommand{\abs}[1]{\left\lvert#1\right\rvert}
\providecommand{\norm}[1]{\left\lVert#1\right\rVert}

\def\beq{\begin{equation}}
\def\eeq{\end{equation}}
\newcommand{\bea}{\begin{eqnarray}}
\newcommand{\eea}{\end{eqnarray}}
\newcommand{\beas}{\begin{eqnarray*}}
\newcommand{\eeas}{\end{eqnarray*}}

{

\author[Yavar Kian]{Yavar Kian}
\address{Univ Rouen Normandie, CNRS, Normandie Univ, LMRS UMR 6085, F-76000 Rouen, France.}
\email{	yavar.kian@univ-rouen.fr}

\author{Faouzi Triki}
\address{F. Triki, Laboratoire Jean Kuntzmann,
UMR CNRS 5224, Universit\'e Grenoble-Alpes, 700 Avenue  Centrale,
38401 Saint-Martin- d'H\`eres, France}
\email{faouzi.triki@univ-grenoble-alpes.fr}

\begin{document}

\begin{abstract}
In  photoacoustic imaging the objective  is to determine the optical properties of biological tissue from boundary measurement of the generated acoustic wave. Here, we propose a restriction to piecewise constant media parameters. Precisely we assume that the acoustic speed and the optical coefficients take two different constants inside and outside a star shaped inclusion. We show that  the inclusion can be uniquely recovered from a single measurement. We also derive a stability estimate of Lipschitz  type of the inversion. The proof of stability is based on an integral representation and a new observability inequality  for the wave equation with piecewise 
constant speed  that is of interest itself. 
\end{abstract}

\title[]{Recovery of an inclusion  in  Photoacoustic imaging}

\maketitle

\section{Introduction}
\label{sec-intro}
\setcounter{equation}{0}
Photoacoustic imaging (PAI)~\cite{ ABGJ1, KuKu-HMMI10,LiWa-PMB09,Wang-Book09, AKK1, CAB, BS} is an  imaging technique  that couples  acoustic and optical waves to achieve high-resolution imaging of optical properties of  biological tissues. In a typical (PAI)  experiment, near infra-red (NIR) photons are radiated  into the biological tissue which is heated up due to the absorption of the  electromagnetic energy. The heating  then results in the expansion of the tissue which causes the generation
of  a pressure wave. The measurement of this later on the boundary is then used to reconstruct  the optical absorption and  diffusion coefficients of the tissue.\\
 
The inversion procedure in (PAI) proceeds in two steps. In the first step, the initial pressure field which  is proportional to the local absorbed energy inside the tissue, is recovered  from pressure wave boundary data. Mathematically speaking, if the acoustic speed is known, this is a linear inverse source problem for the acoustic wave equation~\cite{AgKuKu-PIS09,AmBrJuWa-LNM12, FiHaRa-SIAM07,Haltmeier-SIAM11B,HaScSc-M2AS05,Hristova-IP09,KiSc-SIAM13,KuKu-EJAM08,Kunyansky-IP08,Nguyen-IPI09,PaSc-IP07,QiStUhZh-SIAM11,StUh-IP09}. In the second step, we reconstruct the optical absorption and diffusion coefficients using internal data recovered from  the first inversion~\cite{BaUh-IP10,MaRe-CMS14,NaSc-SIAM14,ReGaZh-SIAM13, BR2, Triki, BR1, Triki2}. \\

 Photoacoustic imaging provides both contrast and resolution. The contrast in (PAI) is mainly due to the sensitivity of the optical absorption and diffusion coefficients  of the tissue in the near infra-red regime. For instance, different biological tissues absorb NIR photons differently. The resolution in (PAI) comes in when the acoustic waves propagate  inside the biological tissue without attenuation, and therefore the corresponding  wavelength  provides a good resolution (usually submillimeter).\\

Due to the lack of information on the spatial distribution within heterogeneous biological tissue and since more than 80$\%$ of the tissue is formed by water  the acoustic speed  in (PAI) is assumed to be a constant equal to  the acoustic speed of the water (such as $1540$ m/s). However, in practice  the true speed changes as the acoustic waves propagate through different types of soft tissues. Indeed experimental studies suggest  that  the acoustic velocity in biological  tissues may have variations of up to $10\%$ of the acoustic velocity in water \cite{Xu, Yoon}.  Failure to compensate the acoustic speed variations in heterogeneous biological tissue leads to aberration artefacts deteriorating the (PAI) image quality.  Few mathematical works already proposed to address this issue by simultaneously determining the acoustic speed and the initial pressure \cite{Stefanov,Knox,Liu,Kian}. However the  inverse problem to recover both the speed and the initial state for wave equations  without additional measurements seem to be severely ill-posed \cite{Stefanov}. Most of the existing results show the uniqueness of solution to the inverse problem by assuming
 additional technical conditions on both the acoustic  speed and the initial pressure \cite{Knox,Liu,Kian}. \\
 
In this work we are interested in solving  the inverse problem with non-smooth parameters.  In particular we assume that the acoustic speed and optical coefficients are piecewise constant functions with the same set of discontinuities. This was established in \cite[Chapter 3]{Vauthrin},  by considering the biological tissue as a mixture of blood and water. The acoustic velocity in water being different from the acoustic velocity in blood, the velocity of the biological tissue then depends on the spatial distribution  of water and blood in the mixture. Furthermore, the optical absorption in the mixture will be mainly caused by hemoglobin (as in a biological tissue), and will therefore also depend on the location of blood in the mixture. We assume here that the blood occupies a bounded inclusion, and the initial state is a function of the characteristic function of such inclusion.  Since the value of the velocity of the blood and water are known the (PAI) inverse problem becomes a problem of 
 identification an inclusion from the measurement of the acoustic wave on the boundary.   \\

The paper is organized as follows. In section 2, we introduce the (PAI) inverse problem and announce the main results of the paper. Assuming that the initial state 
which represents the absorption map of the biological tissue, satisfies some reverse inequality with respect to the set of discontinuity of the acoustic speed, we obtain 
a Lipschitz stability estimate. The main result is summarized in Theorem \ref{t1}.  The well-posedness of the forward problem is studied in section 3. Precisely, the existence  and uniqueness of solution are provided  in Proposition \ref{p1}. In section 4, we derive a new observability inequality for the wave equation with piecewise 
constant speed that is of interest itself. The proof of the stability estimate is based on an integral representation formula of the pressure wave exposed in Theorem  \eqref{t6}. The derivation of this latter is detailed in section 5. Finally, section 6 is devoted to the proof of the main stability estimate of Theorem \ref{t1}.

\section{Statement of the problem and main result }\label{section-main}
In this section we introduce the (PAI)  inverse problem and announce the main results of the paper. \\

Let $\Omega$ be a $C^3$ connected bounded open set of $\R^3$. Let $a$ be a fixed constant within $(\frac{1}{2},1)$ and consider $\omega$ an open set of $\R^3$ such that $\overline{\omega}\subset\Omega$, and define 
\bea \label{speed}
c:=1+(a-1)\mathds{1}_\omega,
\eea
where $\mathds{1}_\omega$ is the characteristic function of the inclusion $\omega$.
We consider the following initial boundary value problem 
\begin{equation}\label{eq1}\left\{\begin{array}{ll}c^{-2}\partial_t^2p-\Delta_x p=0,\quad &\textrm{in}\ (0,+\infty)\times\Omega,\\  p(0,\cdot)= f,\quad \partial_tp(0,\cdot)=g,\quad &\textrm{in}\ \Omega,\\ \partial_\nu p+\beta\partial_tp=0,\quad &\textrm{on}\ (0,+\infty)\times\partial\Omega,\end{array}\right.\end{equation}
where $\beta\in C^2(\partial\Omega;(0,+\infty))$ represents the damping of the acoustic transducers on the boundary. We assume here that $f\in H^2(\Omega)$, $g\in H^1(\Omega)$, and they satisfy the following compatibility conditions
\bel{p1a}\int_\Omega c^{-2}gdx+\int_{\partial\Omega}\beta fd\sigma(x)=0,\quad \partial_\nu f+\beta g=0\textrm{ on }\partial\Omega.
\ee

We will assume here that the initial  condition $f$ depends on the acoustic speed $c$, that is $f = f(c, \cdot) \in  H^2(\Omega)$.\\

Under  assumption \eqref{p1a}, we show in Proposition \ref{p1} the unique existence of a solution $p$ of \eqref{eq1} that lies  in $ C([0,+\infty);H^2(\Omega))\cap C^2([0,+\infty);L^2(\Omega))$,  and we consider  the (PAI) inverse problem   to determine  simultaneously  the sound speed $c$  and the initial pressure $f$ from the knowledge of the pressure data  $p|_{\partial \Omega\times (0, T)}$, with $T>0$ sufficiently large. The objective of this paper is to derive stability estimate for this inverse problem. \\


We fix $M>0$ and we introduce the  subset $\mathcal F$ of  $L^\infty(\Omega)\times H^2(\Omega)\times H^1(\Omega)$ such that for all $(c,f,g)\in\mathcal F$ the following conditions are fulfilled:\\
(i) the map $c$ takes the form \eqref{speed}.\\
(ii)  $g\in H^1(\Omega)$ solves $\Delta g=0$ in $\Omega$ and satisfies the boundary condition $g=-\beta^{-1}\partial_\nu f$ on $\partial\Omega$  \cite{Ren}.\\
(iii) $ c\rightarrow f(c, \cdot)$ is a measurable function on $L^\infty(\Omega)$, and  satisfies the following bound for almost every $c\in L^\infty(\Omega)
$ with $\|c\|_{L^\infty(\Omega)} \leq 1$:
\bel{ass1}\norm{f(c, \cdot)}_{H^2(\Omega)}\leq M.\ee
(iv) For $j=1,2$ and for all $c_j$ of the form \eqref{speed} with $\omega=\omega_j$, there exists $d>0$ independent of $a$ such that
\bel{ass2}\norm{f(c_2, \cdot)-f(c_2, \cdot)}_{H^1(\Omega)}\geq d\norm{\mathds{1}_{\omega_1} -\mathds{1}_{\omega_2}}_{L^\infty(\Omega)},\ee
and we have $f(c_1, \cdot)=f(c_2, \cdot)$ and $\partial_\nu f(c_1, \cdot)=\partial_\nu f(c_2, \cdot)$ on $\partial\Omega$.

\begin{remark}
Notice that since the normal derivative of the initial pressure stays the same for all $c$
of the form \eqref{speed},  we deduce from condition (ii)  that $g$ is independent of $c$. The reverse inequality \eqref{ass2} indiquates that  one can determine
in a stable way the set of discontinuity  of the optical coefficients which is the same as of the acoustic speed from the knowledge of  absorption map
\cite{NaSc-SIAM14}. 
\end{remark}

The following stability estimate  is the  principal result of this article. 

\begin{Thm}\label{t1} For $j=1,2$, let $(c_j,f(c_j,\cdot), g)\in\mathcal F$, with $c_j$ of the form \eqref{speed}   with $\omega=\omega_j$,
and denotes by $p_j$ the solution of \eqref{eq1}. We fix $T=4\textrm{diam}(\Omega)$ and we assume  that $\omega_2$ is star shaped.  Then,    there exists  a constant $a_0=a_0( \beta,  \Omega,M) \in (\frac 1 2, 1)$ such that for any $a\in (a_0, 1) $  we can find  $C= C(\beta,  \Omega,M,a)>0$ such that the following  stability estimates hold

 \bea \label{maineq1}
 \|c_1-c_2\|_{L^\infty(\Omega)} \leq C \|  p_1 -   p_2\|_{H^1((0, T)\times\partial \Omega  )},
 \eea
 \bea \label{maineq2}
 \|f(c_1, \cdot) -f(c_2, \cdot)\|_{H^1(\Omega)} \leq C \left(\|  p_1 -   p_2\|_{H^{\frac{3}{2}}((0, T)\times\partial \Omega  )}+\|  t^{-\frac{1}{2}}(\partial_t p_1 -   \partial_tp_2)\|_{L^2((0, T)\times\partial \Omega  )}\right).
 \eea

\end{Thm}

\begin{remark}
Notice that the assumption $a\in (a_0, 1) $ is in agreement  with the physical setting. Indeed the acoustic speed in the biological tissue   can vary up to  $10\%$ from the constant  acoustic velocity in water \cite{Xu, Yoon}. Here the acoustic speed in the water is normalized to one and  variation in the acoustic speed is bounded by $1-a_0$. 
\end{remark}

\section{Forward problem}

\label{section-forward}
In this section we study the well-posedness of the problem \eqref{eq1}. For equation with smooth coefficients, the well-posedness of  \eqref{eq1} can be deduced by following the analysis of \cite{LR}. Nevertheless, since we deal here with non-smooth sound speeds coefficients $c\in L^\infty(\Omega)$, such properties need to be considered more carefully. In addition, we show in our results that the solution of \eqref{eq1} can be estimated independently of the parameter $a$ appearing in \eqref{speed}.

We denote by $L^2(\Omega;c^{-2}dx)$ the space of measurable functions $h$ satisfying
$$\int_\Omega c^{-2}|h|^2dx<\infty.$$
We associate with $L^2(\Omega;c^{-2}dx)$ the scalar product
$$\left\langle u,v\right\rangle_{L^2(\Omega;c^{-2}dx)}=\int_\Omega c^{-2}u\overline{v}dx$$
and we recall that $L^2(\Omega;c^{-2}dx)$ is an Hilbert space. Using the fact that $0<a\leq c\leq1$, we deduce that $L^2(\Omega;c^{-2}dx)=L^2(\Omega)$ with equivalent norms.
We denote by $\mathcal H$ the linear subspace of $H^1(\Omega)\times L^2(\Omega;c^{-2}dx)$ of  elements of $(u_0,u_1)\in H^1(\Omega)\times L^2(\Omega;c^{-2}dx)$ satisfying the condition
\bel{cond1}\int_\Omega c^{-2}u_1dx+\int_{\partial\Omega}\beta u_0d\sigma(x)=0.\ee
We can prove the following.
\begin{lem}\label{l1} There exists a constant $C>0$ depending only on $\beta$ and $\Omega$ such that, for all $(u_0,u_1)\in \mathcal H$, we have
\bel{l1a}\norm{u_0}_{H^1(\Omega)}^2+\norm{u_1}_{L^2(\Omega;c^{-2}dx)}^2\leq  C\int_\Omega(|\nabla u_0|^2+c^{-2}| u_1|^2)dx.\ee

\end{lem}
\begin{proof} It is clear that the proof will be completed if we can show that 
\bel{l1b}\norm{u_0}_{L^2(\Omega)}^2\leq C\int_\Omega(|\nabla u_0|^2+| u_1|^2)dx.\ee
For this purpose, let us fix $\gamma\in C^1(\overline{\Omega};\R^3)$ and $\tilde{\beta}\in\ C^2(\overline{\Omega})$ satisfying $\gamma=\nu$ and $\tilde{\beta}=\beta$ on $\partial\Omega$. Then, applying the Green formula, we obtain
\bel{l1c}\begin{aligned}\int_{\partial\Omega}\beta u_0d\sigma(x)&=\int_\Omega \textrm{div}(\tilde{\beta} u_0\gamma)dx\\
\ &=\int_\Omega \tilde{\beta} \nabla u_0\cdot\gamma dx+\int_\Omega   u_0\textrm{div}(\tilde{\beta}\gamma) dx.\end{aligned}\ee
Using the fact that $\tilde{\beta}>0$, we deduce that 
$$\int_\Omega   \textrm{div}(\tilde{\beta}\gamma) dx=\int_{\partial\Omega}  \beta\gamma\cdot\nu d\sigma(x)=\int_{\partial\Omega}  \beta d\sigma(x)>0.$$
Therefore, applying the Poincar\'e-Wirtinger inequality, we deduce that
\bel{l1d}\norm{u_0- \frac{\int_\Omega   u_0\textrm{div}(\tilde{\beta}\gamma) dx}{\int_\Omega   \textrm{div}(\tilde{\beta}\gamma) dx}}_{L^2(\Omega)}^2\leq C\int_\Omega|\nabla u_0|^2dx,\ee
with $C>0$ depending only on $\beta$ and $\Omega$. On the other hand, combining \eqref{l1c} with \eqref{cond1}, using the fact that $\frac{1}{2}\leq a\leq c\leq1$ and applying the Cauchy-Schwarz inequality, we obtain
$$\abs{\int_\Omega   u_0\textrm{div}(\tilde{\beta}\gamma) dx}^2\leq 2\left(\abs{\int_\Omega \tilde{\beta} \nabla u_0\cdot\gamma dx}^2+\abs{\int_\Omega c^{-2}u_1dx}^2\right)\leq 
C\int_\Omega(|\nabla u_0|^2+| u_1|^2)dx,$$
with $C>0$ depending only on $\beta$ and $\Omega$. Combining this with \eqref{l1d}, we obtain
$$\norm{u_0}_{L^2(\Omega)}^2\leq  2C\int_\Omega|\nabla u_0|^2dx+2\abs{\frac{\int_\Omega   u_0\textrm{div}(\tilde{\beta}\gamma) dx}{\int_\Omega   \textrm{div}(\tilde{\beta}\gamma) dx}}^2\leq C\int_\Omega(|\nabla u_0|^2+| u_1|^2)dx,$$
with $C>0$ depending only on $\beta$ and $\Omega$. This proves \eqref{l1b} and it completes the proof of the lemma.
\end{proof}
According to Lemma \ref{l1}, one can check that the space $\mathcal H$ endowed with the scalar product 
$$\left\langle (u_0,u_1), (v_0,v_1) \right\rangle=\int_\Omega \nabla u_0\cdot \overline{\nabla v_0}dx+\int_\Omega c^{-2}u_1 \overline{v_1}dx$$
is an Hilbert space. Let us also consider the following density results on $\mathcal H$.

\begin{lem}\label{l11} Consider the space $\mathcal K$ defined by
$$\mathcal K:=\{(u_0,u_1)\in H^2(\Omega)\times H^1_0(\Omega):\ \partial_\nu u_0=0 \textrm{ on }\partial\Omega\}.$$
Then $\mathcal K\cap \mathcal H$ is dense on $\mathcal H$.

\end{lem}
\begin{proof} Let us first show that the $\mathcal K$ is dense on $H^1(\Omega)\times L^2(\Omega;c^{-2}dx)$. Since $H^1_0(\Omega)$ is clearly dense in $L^2(\Omega;c^{-2}dx)$, we only need to prove that the space $\{h\in H^{2}(\Omega):\ \partial_\nu u_0=0 \textrm{ on }\partial\Omega\}$ is dense in $H^1(\Omega)$. For this purpose, consider $-\Delta_N$ the Laplacian with Neumann boundary condition acting on $L^2(\Omega)$ and recall that, since $\Omega$ is $\mathcal C^3$, $D(-\Delta_N)=\{h\in H^{2}(\Omega):\ \partial_\nu u_0=0 \textrm{ on }\partial\Omega\}$.  Then, recalling that $-\Delta_N$ is the unique selfadjoint operator associated with the sesquilinear hermitian form
$$b(u,v)=\int_\Omega \nabla u\cdot\nabla vdx,\quad u,v\in H^1(\Omega),$$
with domain $H^1(\Omega)$,  we deduce that $D(-\Delta_N)$ is dense in $H^1(\Omega)$. This proves that  $\mathcal K$ is dense on $H^1(\Omega)\times L^2(\Omega;c^{-2}dx)$. Now recalling that that $ \mathcal H$ is the kernel of a continuous linear form acting on $H^1(\Omega)\times L^2(\Omega;c^{-2}dx)$, one can check that $\mathcal K\cap \mathcal H$ is dense on $\mathcal H$.
\end{proof}

We denote by $A$ the operator 
$$A=\begin{pmatrix}
0 & 1\\
c^{2}\Delta & 0
\end{pmatrix}$$
acting on the space $\mathcal H$ with domain
$$D(A):=\{(u_0,u_1)\in\mathcal H:\ (u_1,c^{2}\Delta u_0)\in H^1(\Omega)\times L^2(\Omega;c^{-2}dx),\ \partial_\nu u_0+\beta u_1=0 \textrm{ on }\partial\Omega\}.$$
Let us observe that for all $(u_0,u_1)\in D(A)$ we have $A(u_0,u_1)\in\mathcal H$. Indeed, for all $(u_0,u_1)\in D(A)$, fixing $(v_0,v_1)=A(u_0,u_1)$, we obtain
$$\begin{aligned} \int_\Omega c^{-2}v_1dx+\int_{\partial\Omega}\beta v_0d\sigma(x)&=\int_\Omega \Delta u_0dx+\int_{\partial\Omega}\beta u_1d\sigma(x)\\
\ &=\int_\Omega \partial_\nu u_0d\sigma(x)+\int_{\partial\Omega}\beta u_1d\sigma(x)=0.\end{aligned}$$
This proves that  $(v_0,v_1)\in\mathcal H$. Using the properties described above  we will state our first  result of  existence of strong solution of \eqref{eq1} as follows.

\begin{prop}\label{p1} Let $(f,g)\in H^2(\Omega)\times H^1(\Omega)$ be such that the compatibility conditions \eqref{p1a}
are fulfilled. Then problem \eqref{eq1} admits a unique solution $p\in \mathcal C^2([0,+\infty);L^2(\Omega))\cap \mathcal C^1([0,+\infty);H^1(\Omega))\cap \mathcal C([0,+\infty);H^2(\Omega))$ and there exists $C>0$ depending only on $\beta$, $T$ and $\Omega$ such that
\bel{p1bb}\norm{p}_{\mathcal C^2([0,+\infty);L^2(\Omega))}+\norm{p}_{\mathcal C^1([0,+\infty);H^1(\Omega))}+\norm{p}_{\mathcal C([0,+\infty);H^2(\Omega))}\leq C\norm{(f,g)}_{H^2(\Omega)\times H^1(\Omega)}.\ee
\end{prop}

 \begin{proof} This lengthy proof will be divided into four steps.\\
\textbf{Step 1}: in this step we will show that $D(A)$ embedded continuously into $H^2(\Omega)\times H^1(\Omega)$. For this purpose, let us first observe that for $(u_0,u_1)\in D(A)$, we have $u_0,u_1\in H^1(\Omega)$, $\Delta u_0\in L^2(\Omega)$ and $\partial_\nu u_0dx+\beta u_1=0$ on 
$\partial\Omega$. It follows that $u_0\in H^1(\Omega)$, satisfies $\Delta u_0\in L^2(\Omega)$ and $\partial_\nu u_0=-\beta u_1\in H^{\frac{1}{2}}(\partial\Omega)$ and applying \cite[Theorem 2.4.1.3]{Gr} we deduce that $u_0\in H^2(\Omega)$ and, in view of Lemma \ref{l1}, we obtain
$$\begin{aligned}\norm{u_0}_{H^2(\Omega)}+\norm{u_1}_{H^1(\Omega)}&\leq C\left(\norm{\Delta u_0}_{L^2(\Omega)}+\norm{u_0}_{H^1(\Omega)}+\norm{u_1}_{H^1(\Omega)}+\norm{\partial_\nu u_0}_{H^{\frac{1}{2}}(\partial\Omega)}\right)\\
\ &\leq C \left(\norm{c^2\Delta u_0}_{L^2(\Omega;c^{-2}dx)}+\norm{u_0}_{H^1(\Omega)}+\norm{u_1}_{H^1(\Omega)}+\norm{\beta u_1}_{H^{\frac{1}{2}}(\partial\Omega)}\right)\\
\ &\leq C (\norm{c^2\Delta u_0}_{L^2(\Omega;c^{-2}dx)}+\norm{u_0}_{H^1(\Omega)}+\norm{u_1}_{H^1(\Omega)})\\
\ &\leq C(\norm{A(u_0,u_1)}_{\mathcal H}+\norm{(u_0,u_1)}_{\mathcal H})\\
\ &\leq C\norm{(u_0,u_1)}_{D(A)},\end{aligned}$$
with $C>0$ depending only on  $\beta$ and $\Omega$. Here we have used the fact that $\frac{1}{2}\leq a\leq c\leq 1$.
This proves that $D(A)$ embedded continuously into $H^2(\Omega)\times H^1(\Omega)$ and we deduce that
$$D(A):=\{(u_0,u_1)\in\mathcal H\cap H^2(\Omega)\times H^1(\Omega): \partial_\nu u_0+\beta u_1=0 \textrm{ on }\partial\Omega\}.$$

\textbf{Step 2}: In this step we will prove that the operator $A$ generates a semigroup $e^{tA}$.
Recalling that the space $\mathcal K\cap\mathcal H$, with $\mathcal K$ defined in Lemma \ref{l11}, embedded into $D(A)$ and, in view of Lemma \ref{l11}, we deduce that $D(A)$ is dense in $\mathcal H$.  Fix $(u_0,u_1)\in D(A)$ and notice that, integrating by parts, we obtain
\bel{p1b}\begin{aligned}\re\left\langle A(u_0,u_1), (u_0,u_1) \right\rangle&=\re\left(\int_\Omega \nabla u_1\cdot \overline{\nabla u_0}dx+\int_\Omega \Delta u_0 \overline{u_1}dx\right)\\
\ &=\re\left(\int_{\partial\Omega}\partial_\nu u_0 \overline{u_1}d\sigma(x)\right)=-\int_{\partial\Omega}\beta(x)^{-1}|\partial_\nu u_0|^2d\sigma(x)\leq0.\end{aligned}\ee
Thus, applying \cite[Proposition 2.4.2]{CH}, we deduce that $A$ is dissipative. Now let us prove that $A$ is maximal dissipative. For this purpose, according to \cite[Proposition 2.2.6]{CH}, it is enough to show that for any $F=(F_0,F_1)\in\mathcal H$ the equation $Av-v=F$ admits a  solution $v=(v_0,v_1)\in D(A)$. Let us first prove that Ran$(A-Id)$ is a closed subspace of $\mathcal H$. For this purpose, let us consider a sequence $(g_n)_{n\in\mathbb N}$ of Ran$(A-Id)$ that converges in the sense of $\mathcal H$ to $g$. Then, there exists a sequence $(h_n)_{n\in\mathbb N}$ of $D(A)$ such that $g_n=(A-Id)h_n$, $n\in\mathbb N$. Then, for all $m,n\in\mathbb N$, we have
$$\begin{aligned}\norm{g_n-g_m}_{\mathcal H}^2&=\norm{(A-Id)h_n-(A-Id)h_m}_{\mathcal H}^2\\
\ &=\norm{h_n-h_m}_{\mathcal H}^2+\norm{Ah_n-Ah_m}_{\mathcal H}^2-2\re \left\langle h_n-h_m,A(h_n-h_m)\right\rangle.\end{aligned}$$
Combining this with \eqref{p1b}, we deduce that
$$\norm{h_n-h_m}_{\mathcal H}^2+\norm{A(h_n-h_m)}_{\mathcal H}^2\leq \norm{g_n-g_m}_{\mathcal H}^2,\quad m,n\in\mathbb N,$$
which implies that $(h_n)_{n\in\mathbb N}$ is a Cauchy sequence  of $D(A)$. Combining this with the fact that $D(A)$ embedded continuously into $H^2(\Omega)\times H^1(\Omega)$,  we deduce that there exists $h\in D(A)$ such that $(h_n)_{n\in\mathbb N}$ converges to $h$ in the sense of $D(A)$ and we have $g=(A-Id)h$ which proves that Ran$(A-Id)$ is closed. Using this property and applying a density argument, it would be enough to show that for any $F=(F_0,F_1)\in\mathcal H\cap H^2(\Omega)^2$ the equation $Av-v=F$ admits a  solution $v=(v_0,v_1)\in D(A)$. For this purpose, let us observe that the equation $Av-v=F$ can be rewritten as
\begin{equation}\label{p1c}\left\{\begin{array}{ll}c^2\Delta v_0-v_0=F_0+F_1,\quad &\textrm{in}\ \Omega,\\  
v_1=F_0+v_0,\quad &\textrm{in}\ \Omega,\\ 
\partial_\nu v_0+\beta v_0=-\beta F_0,\quad &\textrm{on}\ \partial\Omega.\end{array}\right.\end{equation}
Using the fact that  $\beta\in \mathcal C^2(\partial\Omega;(0,+\infty))$, we deduce that for any $F=(F_0,F_1)\in\mathcal H\cap H^2(\Omega)^2$ the problem \eqref{p1c}, which is an elliptic equation with Robin boundary condition with respect to $v_0$,  admits a solution $v_0,v_1\in H^2(\Omega)$ satisfying $\partial_\nu v_0+\beta v_1=0$ on $\partial\Omega$. Moreover, using the fact that $F=(F_0,F_1)\in\mathcal H\cap H^2(\Omega)^2$, we obtain
$$\begin{aligned}\int_{\partial\Omega}\beta v_0d\sigma(x)&=\int_{\partial\Omega}\beta v_1d\sigma(x)-\int_{\partial\Omega}\beta F_0d\sigma(x)\\
\ &=\int_{\partial\Omega}\beta v_1d\sigma(x)+\int_{\Omega}c^{-2} F_1dx\\
\ &=\int_{\partial\Omega}\beta v_1d\sigma(x)+\int_{\Omega}\Delta v_0dx-\int_{\Omega}c^{-2}(v_0+F_0)dx\\
\ &=\int_{\partial\Omega}(\partial_\nu v_0+\beta v_1)d\sigma(x)-\int_{\Omega}c^{-2}v_1dx\\
\ &=-\int_{\Omega}c^{-2}v_1dx.\end{aligned}$$
This proves that $(v_0,v_1)\in \mathcal H\cap H^2(\Omega)\times H^1(\Omega)$ and $\partial_\nu v_0+\beta v_1=0$ on $\partial\Omega$, which implies that $(v_0,v_1)\in D(A)$. Therefore, $A$ is maximal dissipative with a domain dense in $\mathcal H$ and applying the Hille-Yoshida theorem (see e.g. \cite[Theorem 3.1.1.]{CH}) we deduce that it generates a semigroup $e^{tA}$. 

\textbf{Step 3:} In this step we complete the proof of existence of solutions of \eqref{eq1}. In view of \eqref{p1a} we have $(f,g)\in D(A)$ and, applying \cite[Theorem 3.1.1.]{CH}, we deduce the map $t\mapsto e^{tA}(f,g)\in \mathcal C([0,+\infty);D(A))\cap \mathcal C^1([0,+\infty);\mathcal H)$. Combining this with the continuous embedding $D(A)\subset H^2(\Omega)\times H^1(\Omega)$, we deduce that  problem \eqref{eq1} admits a  solution $p\in \mathcal C^2([0,+\infty);L^2(\Omega))\cap \mathcal C^1([0,+\infty);H^1(\Omega))\cap \mathcal C([0,+\infty);H^2(\Omega))$ given by $(p(t),\partial_tp(t))=e^{tA}(f,g)$. Moreover, applying \cite[Theorem 3.1.1]{CH} and the fact that $D(A)$ embedded continuously into $H^2(\Omega)\times H^1(\Omega)$, we deduce that there exists $C>0$ depending only on  $\beta$, $T$ and $\Omega$ such that
$$\begin{aligned}\norm{p}_{\mathcal C^2([0,+\infty);L^2(\Omega))}+\norm{p}_{\mathcal C^1([0,+\infty);H^1(\Omega))}+\norm{p}_{\mathcal C([0,+\infty);H^2(\Omega))}&\leq C\norm{(p,\partial_tp)}_{\mathcal C([0,+\infty);D(A))}\\
&\leq C\norm{(f,g)}_{D(A)}\\
&\leq  C\norm{(f,g)}_{H^2(\Omega)\cap H^1(\Omega)}.\end{aligned}$$
From this estimate we obtain \eqref{p1bb}.

\textbf{Step 4:} In this step, we prove the uniqueness of solutions of \eqref{eq1}. For this purpose, let us consider the energy $E$ defined by
\bel{energy}E(t)=\int_\Omega c^{-2}|\partial_tp(\cdot,t)|^2+|\nabla p(\cdot,t)|^2dx,\quad t\in[0,+\infty).\ee
Assuming that $p\in \mathcal C^2([0,+\infty);L^2(\Omega))\cap \mathcal C^1([0,+\infty);H^1(\Omega))\cap \mathcal C([0,+\infty);H^2(\Omega))$ solves \eqref{eq1} with $f=g\equiv0$, we deduce that $E\in \mathcal C^1([0,+\infty))$ and, integrating by parts, we get
$$\begin{aligned}E'(t)&=2\re \int_\Omega(c^{-2}\partial_t^2p\overline{\partial_tp}+\nabla p\cdot\nabla\overline{\partial_t p})dx\\
\ &=2\re \int_\Omega(c^{-2}\partial_t^2p-\Delta p)\overline{\partial_tp}dx+2\re\int_{\partial\Omega}\partial_\nu p\overline{\partial_t p}d\sigma(x)\\
\ &=-2\int_{\partial\Omega}\beta|\partial_t p|^2d\sigma(x)\leq0.\end{aligned}$$
Therefore, the map $E$ is non-increasing and, using the fact that $E(0)=0$, we deduce that $E\equiv0$. Thus, $p$ is a constant function and using the fact that $p(\cdot,0)\equiv0$ we deduce that $p\equiv0$. This completes the proof of the uniqueness of solutions $p\in \mathcal C^2([0,+\infty);L^2(\Omega))\cap \mathcal C^1([0,+\infty);H^1(\Omega))\cap \mathcal C([0,+\infty);H^2(\Omega))$ of \eqref{eq1} and it completes the proof of the proposition.

\end{proof}

We can extend this result as follows.
\begin{prop}\label{p2} Let $(f,g)\in H^2(\Omega)\times H^1(\Omega)$ be such that the following condition 

\bel{p2a} \partial_\nu f+\beta g=0\textrm{ on }\partial\Omega,
\ee
is fulfilled. Then problem \eqref{eq1} admits a unique solution $p\in C^2([0,+\infty);L^2(\Omega))\cap C^1([0,+\infty);H^1(\Omega))\cap C([0,+\infty);H^2(\Omega))$. Moreover, there exists  a constant $C>0$ depending on  $\beta$, $\Omega$ such that \eqref{p1bb} is fulfilled.
\end{prop}
\begin{proof} Let us consider $\lambda\in\mathbb C$, defined by
$$\lambda= \frac{\int_\Omega c^{-2}gdx+\int_{\partial\Omega}\beta fd\sigma(x)}{\int_{\partial\Omega}\beta d\sigma(x)}.$$
It is clear that $(f-\lambda,g)\in \mathcal H$ and we deduce that $(f-\lambda,g)\in D(A)$. Therefore, noticing that constant functions solve \eqref{eq1}, with $f$ constant and $g\equiv0$, we deduce that $p=\lambda +v$, with $v\in C^2([0,+\infty);L^2(\Omega))\cap C^1([0,+\infty);H^1(\Omega))\cap C([0,+\infty);H^2(\Omega))$ given by $(v(t),\partial_tv(t))=e^{tA}(f-\lambda,g)$, solves \eqref{eq1} and $p\in C^2([0,+\infty);L^2(\Omega))\cap C^1([0,+\infty);H^1(\Omega))\cap C([0,+\infty);H^2(\Omega))$. For the estimate \eqref{p1bb} and the uniqueness of such solutions of \eqref{eq1} we refer to Proposition \ref{p1}.

\end{proof}

\section{Observability}
\label{section-observ}

In this section, we consider $\omega$ a star-shaped domain and we recall that there exists $x_0\in \omega$ such that for all $x\in \omega$ the interval $[x_0,x]$ is contained into $\omega$. As a consequence of this result, fixing $\bm{n}$ the outward unit normal vector of $\omega$, we find
\bel{bord1}\bm{n}\cdot (x-x_0)\geq0,\quad x\in\partial \omega.\ee
We define also the following constant
\bel{cste}C(x_0)=\sup_{x\in\Omega}|x-x_0|.\ee
Now let us consider the solution of the following initial boundary value problem
\begin{equation}\label{eq2}\left\{\begin{array}{ll}c^{-2}\partial_t^2u-\Delta_x u=F(t,x),\quad &\textrm{in}\ (0,T)\times\Omega,\\  u(0,\cdot)=u_0,\quad \partial_tu(0,\cdot)=u_1,\quad &\textrm{in}\ \Omega,\\  u=0,\quad &\textrm{on}\ (0,T)\times\partial\Omega.\end{array}\right.\end{equation}
with $T>0$, $u_0\in H^1_0(\Omega)$, $u_1\in L^2(\Omega;c^{-2}dx)$, $F\in L^2((0,T)\times\Omega)$ and $c$ of the form \eqref{speed}. Following the argumentation of \cite[Lemma 2.34 and Lemma 2.39]{KKL}, we can prove that \eqref{eq2} admits a unique solution $u\in C([0,T];H^1_0(\Omega))\cap C^1([0,T];L^2(\Omega))$ satisfying $\partial_\nu u|_{(0,T)\times\partial\Omega}\in L^2((0,T)\times\partial\Omega))$. We prove the following observability inequality which is an extension of some known results  to our class of equations with non-smooth coefficients \cite{Zuazua, ACT, ACT2}.

\begin{prop}\label{p3} Let $(u_0,u_1)\in H^1_0(\Omega)\times L^2(\Omega)$ and $T>2C(x_0)a^{-2}$ with $C(x_0)$ defined by \eqref{cste}. Then, for $u$ the solution of \eqref{eq2}, we have 
\bel{p3a} \int_\Omega [|u_1|^2+c^2|\nabla u_0|^2]dx\leq \frac{2C(x_0)}{Ta^2-2C(x_0)}\left(\int_0^T\int_{\partial\Omega}|\partial_\nu u(t,x)|^2d\sigma(x)dt +\int_0^T\int_{\Omega}|F(t,x)|^2dxdt\right)
\ee
is fulfilled. 
\end{prop}
\begin{proof}
By a density argument, it would be enough to prove the result for $u_0,u_1\in \mathcal C^\infty_0(\Omega)$.  With such assumptions it is known that \eqref{eq2} admits a unique solution $u\in \mathcal C^2([0,+\infty);L^2(\Omega))\cap \mathcal C^1([0,+\infty);H^1_0(\Omega))\cap \mathcal C([0,+\infty);H^2(\Omega))$. Without loss of generality we assume also that $u_0,u_1$ are real valued in such way that $u$ is also real valued.
Fixing $x_0\in \omega$ satisfying \eqref{bord1} and multiplying \eqref{eq2} by $\nabla u\cdot(x-x_0)$, we obtain
\bel{p3b}\begin{aligned}\int_0^T\int_\Omega F\nabla u\cdot(x-x_0)dxdt&=\int_0^T\int_\Omega [\partial_t^2u\nabla u\cdot(x-x_0) -c^2\Delta u\nabla u\cdot(x-x_0)]dxdt\\
\ &=\int_0^T\int_\Omega [\partial_t^2u\nabla u\cdot(x-x_0)dxdt-\int_0^T\int_\Omega  c^2\Delta u\nabla u\cdot(x-x_0)dxdt=I_1-I_2.\end{aligned}\ee
From now on we denote by $V$ the vector valued function $V(x)=x-x_0$. Integrating by parts, we obtain
$$\begin{aligned}I_1&=-\int_0^T\int_\Omega \partial_tu\nabla \partial_tu\cdot Vdxdt+\int_\Omega \partial_tu(T,x)\nabla u(T,x)\cdot Vdx-\int_\Omega u_1\nabla u_0\cdot Vdx\\
\ &=-\frac{1}{2}\int_0^T\int_\Omega [\textrm{div}(|\partial_tu|^2 V)-|\partial_tu|^2\textrm{div}(V)]dxdt+\int_\Omega \partial_tu(T,x)\nabla u(T,x)\cdot Vdx-\int_\Omega u_1\nabla u_0\cdot Vdx\\
\ &=-\frac{1}{2}\int_0^T\int_{\partial\Omega} |\partial_tu|^2 V\cdot\nu d\sigma(x)dt+\frac{3}{2}\int_0^T\int_\Omega |\partial_tu|^2dxdt+\int_\Omega \partial_tu(T,x)\nabla u(T,x)\cdot Vdx-\int_\Omega u_1\nabla u_0\cdot Vdx\\
\ &=\frac{3}{2}\int_0^T\int_\Omega |\partial_tu|^2dxdt+\int_\Omega \partial_tu(T,x)\nabla u(T,x)\cdot Vdx-\int_\Omega u_1\nabla u_0\cdot Vdx.\end{aligned}$$
In the same way, we find
$$\begin{aligned}I_2&=\int_0^T\int_{\Omega\setminus \overline{\omega}}  \Delta u\nabla u\cdot Vdxdt+a^2\int_0^T\int_{\omega}  \Delta u\nabla u\cdot Vdxdt\\
\ &=\frac{1}{2}\int_0^T\int_{\Omega\setminus \overline{\omega}}  [\textrm{div}( |\nabla u|^2 V)-|\nabla u|^2\textrm{div}(V)]dxdt+\frac{a^2}{2}\int_0^T\int_{\omega}  [\textrm{div}( |\nabla u|^2 V)-|\nabla u|^2\textrm{div}(V)] dxdt\\
\ &=\frac{1}{2}\int_0^T\int_{\Omega\setminus \overline{\omega}}  \textrm{div}( |\nabla u|^2 V)dxdt-\frac{3}{2}\int_0^T\int_{\Omega\setminus \overline{\omega}}  |\nabla u|^2 dxdt+\frac{a^2}{2}\int_0^T\int_{\omega}  \textrm{div}( |\nabla u|^2 V)dxdt-\frac{3a^2}{2}\int_0^T\int_{\omega}  |\nabla u|^2 dxdt\\
\ &=\frac{(a^2-1)}{2}\int_0^T\int_{\partial \omega} |\nabla u|^2 V\cdot \bm{n} d\sigma(x)dt+\frac{1}{2}\int_0^T\int_{\partial \Omega} |\nabla u|^2 V\cdot \nu d\sigma(x)dt-\frac{3}{2}\int_0^T\int_{\Omega}  c^2|\nabla u|^2 dxdt\\
\ &=\frac{(a^2-1)}{2}\int_0^T\int_{\partial \omega} |\partial_{\bm{n}} u|^2 V\cdot \bm{n} d\sigma(x)dt+\frac{1}{2}\int_0^T\int_{\partial \Omega} |\partial_\nu u|^2 V\cdot \nu d\sigma(x)dt-\frac{3}{2}\int_0^T\int_{\Omega}  c^2|\nabla u|^2 dxdt.\end{aligned}$$
Combining this with \eqref{p3b}, we find
$$\begin{aligned}\int_0^T\int_\Omega F\nabla u\cdot(x-x_0)dxdt&=I_1-I_2\\
&=\frac{3}{2}\int_0^T\int_\Omega [|\partial_tu|^2+c^2|\nabla u|^2]dxdt+\int_\Omega \partial_tu(T,x)\nabla u(T,x)\cdot Vdx-\int_\Omega u_1\nabla u_0\cdot Vdx\\
\ &\ \ \ -\frac{(a^2-1)}{2}\int_0^T\int_{\partial\Omega} |\partial_\nu u|^2 V\cdot \bm{n} d\sigma(x)dt-\frac{1}{2}\int_0^T\int_{\partial \Omega} |\partial_\nu u|^2 V\cdot \nu d\sigma(x)dt\end{aligned}$$
and we obtain
$$\begin{aligned}&\frac{3}{2}\int_0^T\int_\Omega [|\partial_tu|^2+c^2|\nabla u|^2]dxdt+\int_\Omega \partial_tu(T,x)\nabla u(T,x)\cdot Vdx-\int_\Omega u_1\nabla u_0\cdot Vdx\\
&=-\frac{(1-a^2)}{2}\int_0^T\int_{\partial\Omega} |\partial_{\bm{n}} u|^2 V\cdot \bm{n} d\sigma(x)dt+\frac{1}{2}\int_0^T\int_{\partial \Omega} |\partial_\nu u|^2 V\cdot \nu d\sigma(x)dt+\int_0^T\int_\Omega F\nabla u\cdot(x-x_0)dxdt.\end{aligned}$$
Using the fact that $a\in(0,1)$ and applying \eqref{bord1}, we obtain that 
$$\frac{(1-a^2)}{2}\int_0^T\int_{\partial\Omega} |\partial_{\bm{n}} u|^2 V\cdot \bm{n} d\sigma(x)dt\geq0$$
and it follows that
$$\begin{aligned}&\frac{3}{2}\int_0^T\int_\Omega [|\partial_tu|^2+c^2|\nabla u|^2]dxdt+\int_\Omega \partial_tu(T,x)\nabla u(T,x)\cdot Vdx-\int_\Omega u_1\nabla u_0\cdot Vdx+\\
&\leq\frac{1}{2}\int_0^T\int_{\partial \Omega} |\partial_\nu u|^2 V\cdot \nu d\sigma(x)dt+2C(x_0)\int_0^T\int_\Omega |F|^2dxdt+\frac{1}{8}\int_0^T\int_\Omega |\nabla u|^2dxdt.\end{aligned}$$
Recalling that $c^2\geq a^2\geq \frac{1}{4}$, we obtain
\bel{p3c}\begin{aligned}&\int_0^T\int_\Omega [|\partial_tu|^2+c^2|\nabla u|^2]dxdt+\int_\Omega \partial_tu(T,x)\nabla u(T,x)\cdot Vdx-\int_\Omega u_1\nabla u_0\cdot Vdx+\\
&\leq\frac{1}{2}\int_0^T\int_{\partial \Omega} |\partial_\nu u|^2 V\cdot \nu d\sigma(x)dt+2C(x_0)\int_0^T\int_\Omega |F|^2dxdt.\end{aligned}.\ee
Fixing 
$$E(t):=\int_\Omega [|\partial_tu(t)|^2+c^2|\nabla u(t)|^2]dx,\quad t\in[0,T],$$
we deduce that $E\in \mathcal C^1([0,T])$ and $E'\equiv0$. It follows that
$$\int_0^T\int_\Omega [|\partial_tu|^2+c^2|\nabla u|^2]dxdt=\int_0^TE(t)dt=TE(0)=T\int_\Omega [|u_1|^2+c^2|\nabla u_0|^2]dx.$$
On the other hand, using the fact that $c\geq a$, we find
$$\begin{aligned}&\abs{\int_\Omega \partial_tu(T,x)\nabla u(T,x)\cdot Vdx-\int_\Omega u_1\nabla u_0\cdot Vdx}\\
\ &\leq\int_\Omega \abs{\partial_tu(T,x)}\abs{\nabla u(T,x)}\abs{V}+\int_\Omega \abs{u_1}\abs{\nabla u_0}\abs{ V}dx\\
\ &\leq C(x_0)a^{-2}(E(T)+E(0))=2C(x_0)a^{-2}\int_\Omega [|u_1|^2+c^2|\nabla u_0|^2]dx.\end{aligned}$$
Combining these estimates with \eqref{p3c}, we obtain
$$\begin{aligned}\left(T-2C(x_0)a^{-2}\right)\int_\Omega [|u_1|^2+c^2|\nabla u_0|^2]dx&\leq \int_0^T\int_{\partial \Omega} |\partial_\nu u|^2 |V| d\sigma(x)dt+2C(x_0)\int_0^T\int_\Omega |F|^2dxdt\\
&\leq C(x_0)\int_0^T\int_{\partial \Omega} |\partial_\nu u|^2  d\sigma(x)dt+2C(x_0)\int_0^T\int_\Omega |F|^2dxdt.\end{aligned}$$
This last estimate clearly implies \eqref{p3a}.\end{proof}

\section{Representation Formula}
Let the condition of Theorem \ref{t1} be fulfilled with $(c_j,f(c_j, \cdot), g)\in\mathcal F$, and denote  $f_j = f(c_i, \cdot)$.
 Recall that  there exists a constant $C$ depending on $\beta$ and $\Omega$ such that
\begin{equation}\label{t6aa}\norm{g}_{H^1(\Omega)}\leq C\norm{f_j}_{H^2(\Omega)} \end{equation}
and $(f_j,g)\in H^2(\Omega)\cap H^1(\Omega)$ satisfies the compatibility condition \eqref{p2a} with $(f,g)=(f_j,g_j)$, $j=1,2$. Let us consider $p_j\in  C^2([0,+\infty);L^2(\Omega))\cap C^1([0,+\infty);H^1(\Omega))\cap C([0,+\infty);H^2(\Omega))$ the solution of \eqref{eq1} with $c=c_j$ and $(f,g)=(f_j,g_j)$.  Consider $p=p_2-p_1$ and notice that $p$ solves the problem
\begin{equation}\label{eq3}\left\{\begin{array}{ll}\partial_t^2p-c_2^{2}\Delta_x p=c_2^{2}(c_1^{-2}-c_2^{-2})\partial_t^2p_1,\quad &\textrm{in}\ (0,+\infty)\times\Omega,\\  p(0,\cdot)=f,\quad \partial_tp(0,\cdot)=g,\quad &\textrm{in}\ \Omega,\\ \partial_\nu p+\beta\partial_tp=0,\quad &\textrm{on}\ (0,+\infty)\times\partial\Omega,\end{array}\right.\end{equation}
with $f=f_1-f_2$ and $g=g_1-g_2$.
The main result of this section can be stated as follows.

\begin{Thm}\label{t6}
Fix $T=4\textrm{diam}(\Omega)$, $a\in\left(\frac{3}{4},1\right)$ and define the space $\vphantom{H^1((0,T)\times\partial\Omega)}_0 H^1((0,T)\times\partial\Omega):=\{h\in H^1((0,T)\times\partial\Omega):\ h(0,\cdot)\equiv0\}$. Then, there exist three bounded linear operators $\mathcal G\in\mathcal B(L^\infty(\Omega);H^{1}_0(\Omega))$,  $\mathcal I\in\mathcal B(L^2((0,T)\times\partial\Omega);H^{1}_0(\Omega))$ and $\mathcal J\in\mathcal B(\vphantom{H^1((0,T)\times\partial\Omega)}_0 H^1((0,T)\times\partial\Omega);H^{1}_0(\Omega))$ such that the identity
\begin{equation}\label{t6a}\mathcal G[c_2^{2}(c_1^{-2}-c_2^{-2})]-f= \mathcal I[\beta\partial_t p|_{(0,T)\times\partial\Omega}] +\mathcal J [p|_{(0,T)\times\partial\Omega}] \end{equation}
holds true. Here the operators $\mathcal G$, $\mathcal I$, $\mathcal J$ are chosen in such a way that  there exists a constant $C>0$ depending on  $\Omega$,  and $\beta$ such that
\begin{equation}\label{t6b}\norm{\mathcal G}_{\mathcal B(L^\infty(\Omega);H^{1}_0(\Omega))}\leq C\norm{f_1}_{H^2(\Omega)},\end{equation}
\begin{equation}\label{t6bbb}\norm{\mathcal I}_{\mathcal B(L^2((0,T)\times\partial\Omega);H^{1}_0(\Omega))}+\norm{\mathcal J}_{\mathcal B(\vphantom{H^1((0,T)\times\partial\Omega)}_0 H^1((0,T)\times\partial\Omega);H^{1}_0(\Omega))}\leq C.\end{equation}
\end{Thm}

The proof of Theorem \ref{t6} is based on the construction of a boundary operator that we build by applying the observability inequality of Proposition  \ref{p3}. For this purpose, let us first recall the definition  of solutions of \eqref{eq2}, when $F\equiv0$, in the transposition sense. For this purpose, let us consider the unbounded operator $L$ acting on $L^2(\Omega;c^{-2}dx)$ with domain $D(L)=H^1_0(\Omega)\cap H^2(\Omega)$ defined by $Lh=-c^2 \Delta h$, $h\in D(L)$. It is well known that $L$ is a strictly positive selfadjoint operator with a compact resolvent. Therefore, the spectrum of $L$ consists of a non-decreasing sequence of strictly positive eigenvalues $(\lambda_k)_{k\geq1}$. Let us also introduce 
an orthonormal  basis in the Hilbert space $L^2(\Omega;c^{-2} dx)$ of eigenfunctions $(\phi_k)_{k\geq1}$ of $L$ associated with the non-decreasing sequence of  eigenvalues $(\lambda_k)_{k\geq1}$. For all $s\geq 0$, we denote by $L^s$ the operator defined by 
$$L^s g=\sum_{k=1}^{+\infty}\left\langle g,\phi_k\right\rangle \lambda_{k}^s\phi_{k},\quad g\in D(L^s)=\left\{h\in L^2(\Omega):\ \sum_{k=1}^{+\infty}\abs{\left\langle g,\phi_{k}\right\rangle}^2 \lambda_{k}^{2s}<\infty\right\}$$
and consider on $D(L^s)$ the norm
\[\|g\|_{D(L^s)}=\left(\sum_{k=1}^{+\infty}\abs{\left\langle g,\phi_{k}\right\rangle}^2 \lambda_{k}^{2s}\right)^{\frac{1}{2}},\quad g\in D(L^s).\]
We can also set $D(L^{-s})$ the dual space of $D(L^s)$, with respect to $L^2(\Omega;c^{-2} dx)$,  which is a Hilbert space with 
the norm 
\[\norm{v}_{D(L^{-s})}=\left(\sum_{k=1}^\infty \abs{\left\langle v,\phi_{k}\right\rangle_{-2s}}^2\lambda_{k}^{-2s}\right)^{\frac{1}{2}}.\]
Let us recall that $D(L^{\frac{1}{2}})=H^1_0(\Omega)$ with equivalent norm but due to the fact that the coefficient $c$ is not smooth the space $D(L^{-\frac{1}{2}})$ do not coincide with $H^{-1}(\Omega)$. In addition, one can check that
$$\norm{v}_{D(L^{\frac{1}{2}})}^2=\norm{\nabla v}_{L^2(\Omega)}^2,\quad v\in H^1_0(\Omega)$$
and deduce that there exists $C>1$ deppending only on $\Omega$ such that
\bel{L}C^{-1}\norm{v}_{H^1(\Omega)}\leq\norm{v}_{D(L^{\frac{1}{2}})}\leq C\norm{v}_{H^1(\Omega)},\quad v\in H^1_0(\Omega).\ee

Fixing $\psi_0\in L^2(\Omega;c^{-2} dx)$, $\psi_1\in D(L^{-\frac{1}{2}})$ and $g\in L^2((0,T)\times\partial\Omega)$, we consider the solution in the transposition sense of
\begin{equation}\label{eq112}\left\{\begin{array}{ll}\partial_t^2\psi-c^2_2\Delta_x \psi=0,\quad &\textrm{in}\ (0,T)\times\Omega,\\  \psi(0,\cdot)=\psi_0,\quad \partial_t\psi(0,\cdot)=\psi_1,\quad &\textrm{in}\ \Omega,\\  \psi=g,\quad &\textrm{on}\ (0,T)\times\partial\Omega.\end{array}\right.\end{equation}
More precisely, for any $F\in L^1(0,T;L^2(\Omega))$ let us consider the solution $v_F\in \mathcal C([0,T];H^1_0(\Omega))\cap \mathcal C^1([0,T];L^2(\Omega))$ of 
$$\left\{\begin{array}{ll}\partial_t^2v_F-c^2_2\Delta_x v_F=F,\quad &\textrm{in}\ (0,T)\times\Omega,\\  v_F(T,\cdot)=0,\quad \partial_tv_F(T,\cdot)=0,\quad &\textrm{in}\ \Omega,\\  v_F=0,\quad &\textrm{on}\ (0,T)\times\partial\Omega.\end{array}\right.$$
The solution in the transposition sense of \eqref{eq112}, is the unique element $\psi\in L^\infty(0,T;L^2(\Omega;c^{-2}dx))$ satisfying
\bel{trans1a}\begin{aligned}&\int_0^T\left\langle \psi(t,\cdot),F(t,\cdot)\right\rangle_{L^2(\Omega;c^{-2}dx)} dt\\
&=-\left\langle \psi_0,\partial_tv_F(0,\cdot)\right\rangle_{L^2(\Omega;c^{-2}dx)}+\left\langle \psi_1,v_F(0,\cdot)\right\rangle_{D(L^{-\frac{1}{2}}),D(L^{\frac{1}{2}})}-\int_0^T\int_{\partial\Omega}\partial_\nu v_F(t,x)g(t,x)d\sigma(x)dt.\end{aligned}\ee
In a similar way to \cite[Corollary 2.36]{KKL}, we can show that \eqref{eq112} admits a unique  solution $\psi\in C([0,T];L^2(\Omega;c^{-2}dx))\cap C^1([0,T]:D(L^{-\frac{1}{2}}))$ in the transposition sense.

Let us consider the problem
\begin{equation}\label{eq6}\left\{\begin{array}{ll}\partial_t^2\phi-c_2^{2}\Delta_x \phi=0,\quad &\textrm{in}\ (0,T)\times\Omega,\\  \phi(0,\cdot)=0,\quad \partial_t\phi(0,\cdot)=\phi_0,\quad &\textrm{in}\ \Omega,\\  \phi(T,\cdot)=0,\quad \partial_t\phi(T,\cdot)=0,\quad &\textrm{in}\ \Omega,\\ \phi=\Lambda \phi_0,\quad &\textrm{on}\ (0,T)\times\partial\Omega,\end{array}\right.\end{equation}
with $\Lambda$ a suitable control operator from $D(L^{-\frac{1}{2}})$ to $L^2((0,T)\times\partial\Omega)$. We recall that $c_2$ takes the form \eqref{speed} with $\omega=\omega_2$ a star-shapped domain. Therefore, applying the observability inequality of Proposition \ref{p3}, we can prove the following result.
\begin{lem}\label{l4} For $T=4\textrm{diam}(\Omega)$ and $a\in\left(\frac{3}{4},1\right)$, we can define $\Lambda\in \mathcal B(D(L^{-\frac{1}{2}});L^2((0,T)\times\partial\Omega))$ such that for any $\phi_0\in D(L^{-\frac{1}{2}})$, problem \eqref{eq6} admits a unique solution $\phi\in \mathcal C([0,T];L^2(\Omega))\cap \mathcal C^1([0,T];D(L^{-\frac{1}{2}}))$ in the transposition sense satisfying
\begin{equation}\label{l4a}\norm{\phi}_{L^\infty(0,T;L^2(\Omega))}\leq C\norm{\phi_0}_{D(L^{-\frac{1}{2}})}\end{equation}
with $C>0$ depending only on $\Omega$.
\end{lem}
\begin{proof}
Let us start by proving the existence of a solution of problem \eqref{eq6}  in the transposition sense.
Let $\phi_0\in D(L^{-\frac{1}{2}})$, $g\in L^2((0,T)\times\partial\Omega)$, $F\in L^1(0,T;L^2(\Omega))$, $(u_0,u_1)\in H^1_0(\Omega)\times L^2(\Omega)$ and consider the solutions of the following initial boundary value problems
\begin{equation}\label{l4b}\left\{\begin{array}{ll}\partial_t^2v-c_2^{-2}\Delta_x v=F,\quad &\textrm{in}\ (0,T)\times\Omega,\\  v(T,\cdot)=0,\quad \partial_tv(T,\cdot)=0,\quad &\textrm{in}\ \Omega,\\  u=0,\quad &\textrm{on}\ (0,T)\times\partial\Omega,\end{array}\right.\end{equation}
\begin{equation}\label{l4c}\left\{\begin{array}{ll}\partial_t^2w-c_2^2\Delta_x w=0,\quad &\textrm{in}\ (0,T)\times\Omega,\\  w(T,\cdot)=u_0,\quad \partial_tw(T,\cdot)=u_1,\quad &\textrm{in}\ \Omega,\\  w=0,\quad &\textrm{on}\ (0,T)\times\partial\Omega,\end{array}\right.\end{equation} 
\begin{equation}\label{l4d}\left\{\begin{array}{ll}\partial_t^2\psi_1-c_2^2\Delta_x \psi_1=0,\quad &\textrm{in}\ (0,T)\times\Omega,\\  \psi_1(0,\cdot)=0,\quad \partial_t\psi_1(0,\cdot)=\phi_0,\quad &\textrm{in}\ \Omega,\\  \psi_1=0,\quad &\textrm{on}\ (0,T)\times\partial\Omega,\end{array}\right.\end{equation}
\begin{equation}\label{l4e}\left\{\begin{array}{ll}\partial_t^2\psi_2-c_2^2\Delta_x \psi_2=0,\quad &\textrm{in}\ (0,T)\times\Omega,\\  \psi_2(0,\cdot)=0,\quad \partial_t\psi_2(0,\cdot)=0,\quad &\textrm{in}\ \Omega,\\  \psi_2=g,\quad &\textrm{on}\ (0,T)\times\partial\Omega.\end{array}\right.\end{equation}
Let us recall that problem \eqref{l4d} admits a unique solution in the transposition sense $\psi_1\in \mathcal C([0,T];L^2(\Omega))\cap \mathcal C^1([0,T];D(L^{-\frac{1}{2}}))$ which is the unique element of $L^\infty(0,T;L^2(\Omega;c^{-2}dx))$ satisfying,  for all $F\in L^1(0,T;L^2(\Omega))$, the following identity
\begin{equation}\label{l4f}\int_0^T\int_\Omega c_2^{-2}\psi_1 Fdxdt=\left\langle \phi_0,v(0,\cdot)\right\rangle_{D(L^{-\frac{1}{2}}),D(L^{\frac{1}{2}})},\end{equation}
with $v\in \mathcal C([0,T];H^1_0(\Omega))\cap \mathcal C^1([0,T];L^2(\Omega))$ the solution of \eqref{l4b}.
In the same way,  problem \eqref{l4e} admits a unique solution in the transposition sense $\psi_2\in \mathcal C([0,T];L^2(\Omega))\cap \mathcal C^1([0,T];D(L^{-\frac{1}{2}}))$ which is the unique element of $L^\infty(0,T;L^2(\Omega))$ satisfying  for all $F\in L^1(0,T;L^2(\Omega))$ 
\begin{equation}\label{l4g}\int_0^T\int_\Omega c_2^{-2}\psi_2 Fdxdt=-\int_0^T\int_{\partial\Omega}g\partial_\nu vd\sigma(x)dt,\end{equation}
with $v\in \mathcal C([0,T];H^1_0(\Omega))\cap \mathcal C^1([0,T];L^2(\Omega))$ the solution of \eqref{l4b}.
Since $T>2Diam(\Omega)a^{-2}$, Proposition \ref{p3} and the Poincarr\'e inequality imply that for any $(u_0,u_1)\in H^1_0(\Omega)\times L^2(\Omega)$ we have
$$\norm{u_1}_{L^2(\Omega)}+\norm{u_0}_{H^1(\Omega)}\leq C_1\left(\frac{2Diam(\Omega)}{(Ta^2-2Diam(\Omega))}\right)^{\frac{1}{2}}\norm{\partial_\nu w}_{L^2((0,T)\times\partial\Omega)},$$
with $C_1>0$ depending only on $\Omega$. Using the fact that $a>\frac{3}{4}$ and $T=4\textrm{diam}(\Omega)$, we deduce that $Ta^2-2Diam(\Omega))\geq \frac{Diam(\Omega))}{4}>0$ and we get
\begin{equation}\label{l4h}\norm{u_1}_{L^2(\Omega)}+\norm{u_0}_{H^1(\Omega)}\leq 4C_1\norm{\partial_\nu w}_{L^2((0,T)\times\partial\Omega)},\end{equation}
 Therefore, applying the Hahn Banach theorem, for any $\phi_0\in D(L^{-\frac{1}{2}})$ we can define $\Lambda \phi_0\in L^2((0,T)\times\partial\Omega)$ such that
\begin{equation}\label{l4i}\int_0^T\int_{\partial\Omega}\Lambda \phi_0\partial_\nu wd\sigma(x)dt=\left\langle \psi_1(T,\cdot),u_1\right\rangle_{L^2(\Omega;c^{-2}dx)}-\left\langle \partial_t\psi_1(T,\cdot),u_0\right\rangle_{D(L^{-\frac{1}{2}}),D(L^{\frac{1}{2}})}.\end{equation}
 Now let us consider $\mathcal N$ to be the continuous linear form on $L^1(0,T;L^2(\Omega))\times H^1_0(\Omega)\times L^2(\Omega)$ given by
$$\mathcal N(F,u_0,u_1)=-\int_0^T\int_{\partial\Omega}\Lambda \phi_0\partial_\nu vd\sigma(x)dt-\int_0^T\int_{\partial\Omega}\Lambda \phi_0\partial_\nu wd\sigma(x)dt.$$
Then, there exists $\psi_2\in L^\infty(0,T;L^2(\Omega))$, $h_1\in L^2(\Omega;c^{-2}dx)$, $h_2\in D(L^{-\frac{1}{2}})$ such that
$$\begin{aligned}&\int_0^T\int_\Omega \psi_2 Fdxdt+\left\langle h_1,u_1\right\rangle_{L^2(\Omega)}+\left\langle h_2,u_0\right\rangle_{D(L^{-\frac{1}{2}}),D(L^{\frac{1}{2}})}\\
&=-\int_0^T\int_{\partial\Omega}\Lambda \phi_0\partial_\nu vd\sigma(x)dt-\int_0^T\int_{\partial\Omega}\Lambda \phi_0\partial_\nu wd\sigma(x)dt.\end{aligned}$$
Fixing $u_0=u_1=0$, we deduce that $\psi_2\in L^\infty(0,T;L^2(\Omega))$ satisfies \eqref{l4g} and from the uniqueness of this expression, we deduce that  $\psi_2\in \mathcal C([0,T];L^2(\Omega))\cap \mathcal C^1([0,T];D(L^{-\frac{1}{2}}))$ is the unique solution in the transposition sense of \eqref{l4e} with $g=\Lambda\phi_0$. In the same way, fixing $F=0$, we deduce that $h_1=\psi_2(T,\cdot)$ and $h_2=-\partial_t\psi_2(T,\cdot)$ and it follows
$$\left\langle \psi_2(T,\cdot),u_1\right\rangle_{L^2(\Omega;c^{-2}dx)}-\left\langle \partial_t\psi_2(T,\cdot),u_0\right\rangle_{D(L^{-\frac{1}{2}}),D(L^{\frac{1}{2}})}=-\int_0^T\int_{\partial\Omega}\Lambda \phi_0\partial_\nu wd\sigma(x)dt.$$
Combining this with \eqref{l4i}, we deduce that $\psi_2(T,\cdot)=-\psi_1(T,\cdot)$ and $\partial_t\psi_2(T,\cdot)=-\partial_t\psi_1(T,\cdot)$ and $\phi=\psi_1+\psi_2$ solves \eqref{eq6} in the sense of transposition. This completes the proof of the first statement of Lemma \ref{l4a}. In order to complete the proof of the lemma, we need to prove the estimate \eqref{l4a}. For this purpose, let us first observe that using the fact that $\frac{1}{2}\leq a\leq c_2\leq1$ and applying classical energy estimates, one can check that there exists a constant $C_2>0$ depending only on  $\Omega$ such that
for  $v\in \mathcal C([0,T];H^1_0(\Omega))\cap \mathcal \mathcal C^1([0,T];L^2(\Omega))$ solving  \eqref{l4b} and $w\in  C([0,T];H^1_0(\Omega))\cap  C^1([0,T];L^2(\Omega))$ solving  \eqref{l4c}, we have
\begin{equation}\label{l4j}\norm{v}_{ C([0,T];H^1_0(\Omega))}+\norm{v}_{ C^1([0,T];L^2(\Omega))}+\norm{\partial_\nu v}_{L^2((0,T)\times\partial\Omega)}\leq C_2\norm{F}_{L^1(0,T;L^2(\Omega))},\end{equation}
\begin{equation}\label{l4k}\norm{w}_{ C([0,T];H^1_0(\Omega))}+\norm{w}_{ C^1([0,T];L^2(\Omega))}+\norm{\partial_\nu w}_{L^2((0,T)\times\partial\Omega)}\leq C_2(\norm{u_0}_{H^1(\Omega)}+\norm{u_1}_{L^2(\Omega)}).\end{equation}
In the same way, following the argumentation in \cite[Lemma 2.42]{KKL}, we can prove that there exists a constant $C_3>0$ depending only on  $\Omega$  such that
\begin{equation}\label{l4l}\norm{\psi_1}_{ C([0,T];L^2(\Omega))}+\norm{\psi_1}_{ C^1([0,T];D(L^{-\frac{1}{2}}))}\leq C_3\norm{\phi_0}_{D(L^{-\frac{1}{2}})}.\end{equation}
Combining formula \eqref{l4h}-\eqref{l4i} with \eqref{l4l}, we deduce that
$$\begin{aligned}&\abs{\int_0^T\int_{\partial\Omega}\Lambda \phi_0\partial_\nu wd\sigma(x)dt}\\
&\leq \norm{\psi_1}_{ C([0,T];L^2(\Omega))}\norm{u_0}_{L^2(\Omega)}+\norm{\psi_1}_{ C^1([0,T];D(L^{-\frac{1}{2}}))}\norm{u_1}_{H^{1}(\Omega)}\\
&\leq C_3\norm{\phi_0}_{D(L^{-\frac{1}{2}})}(\norm{u_0}_{L^2(\Omega)}+\norm{u_1}_{H^{1}(\Omega)})\\
&\leq 4C_3C_1\norm{\partial_\nu w}_{L^2((0,T)\times\partial\Omega)}\norm{\phi_0}_{D(L^{-\frac{1}{2}})}.\end{aligned}$$
Thus, we have $\Lambda\in\mathcal B(D(L^{-\frac{1}{2}});L^2((0,T)\times\partial\Omega))$ with
\begin{equation}\label{l4m}\norm{\Lambda}_{\mathcal B(D(L^{-\frac{1}{2}});L^2((0,T)\times\partial\Omega))}\leq 4C_3C_1.\end{equation}
Now recall that for any $F\in L^1(0,T;L^2(\Omega))$ and $v\in  C([0,T];H^1_0(\Omega))\cap  C^1([0,T];L^2(\Omega))$ solving  \eqref{l4b}, we have
$$\int_0^T\int_\Omega c^{-2}_2\psi_2 Fdxdt=-\int_0^T\int_{\partial\Omega}\Lambda \phi_0\partial_\nu vd\sigma(x)dt$$
and applying \eqref{l4j} and \eqref{l4m}, we obtain
$$\begin{aligned}\abs{\int_0^T\int_\Omega c_2^{-2}\psi_2 Fdxdt}&\leq \norm{\Lambda \phi_0}_{L^2((0,T)\times\partial\Omega)}\norm{\partial_\nu v}_{L^2((0,T)\times\partial\Omega)}\\
&\leq 4C_3C_1\norm{\phi_0}_{D(L^{-\frac{1}{2}})}C_2\norm{F}_{L^1(0,T;L^2(\Omega))}.\end{aligned}$$
Therefore, we find
$$\norm{\psi_2}_{L^\infty(0,T;L^2(\Omega))}\leq 4C_3C_2C_1\norm{\phi_0}_{D(L^{-\frac{1}{2}})}.$$
Combining this with \eqref{l4l} and recalling that $\phi=\psi_1+\psi_2$ we deduce \eqref{l4a}. This completes the proof of the lemma.\end{proof}

Applying Lemma \ref{l4}, we are now in position to complete the proof of Theorem \ref{t6}.

 \textbf{Proof of Theorem \ref{t6}.} Let us first recall that, following the arguments of \cite[Theorem 2.1]{LLT} (see also \cite[Lemma 2.43]{KKL}), for all $v_0\in H^1_0(\Omega)$, $g\in  \vphantom{H^1((0,T)\times\partial\Omega)}_0 H^1((0,T)\times\partial\Omega)$, $F\in L^1(0,T;L^2(\Omega))$ one can check that the problem
$$\left\{\begin{array}{ll}\partial_t^2v-c_2^2\Delta_x v=F,\quad &\textrm{in}\ (0,T)\times\Omega,\\  v(0,\cdot)=v_0,\quad \partial_tv(0,\cdot)=0,\quad &\textrm{in}\ \Omega,\\  v=g,\quad &\textrm{on}\ (0,T)\times\partial\Omega,\end{array}\right.$$
admits a unique solution $v\in C^1([0,T];L^2(\Omega))\cap C([0,T];H^1(\Omega))$ satisfying $\partial_\nu v\in L^2((0,T)\times\partial\Omega)$
and the estimate
$$\norm{v}_{C^1([0,T];L^2(\Omega))}+\norm{v}_{C([0,T];H^1(\Omega))}+\norm{\partial_\nu v}_{L^2((0,T)\times\partial\Omega)}\leq C(\norm{u_0}_{H^1(\Omega)}+\norm{F}_{L^1(0,T;L^2(\Omega))}+\norm{g}_{H^1((0,T)\times\partial\Omega)}),$$
with $C>0$ depending on $\Omega$ and $\beta$. Therefore, applying Lemma \ref{l4}, we can consider $\phi\in C([0,T];L^2(\Omega))$ solving \eqref{eq6} and  we can define the normal derivative of $\phi$ as an element of the dual space $\vphantom{H^1((0,T)\times\partial\Omega)}_0 H^1((0,T)\times\partial\Omega)'$  of $\vphantom{H^1((0,T)\times\partial\Omega)}_0 H^1((0,T)\times\partial\Omega)$ given by 
$$\left\langle \partial_\nu\phi,g\right\rangle_{\vphantom{H^1((0,T)\times\partial\Omega)}_0 H^1((0,T)\times\partial\Omega)',\vphantom{H^1((0,T)\times\partial\Omega)}_0 H^1((0,T)\times\partial\Omega)}=-\left\langle \phi_0,v_0\right\rangle_{D(L^{-\frac{1}{2}}),D(L^{\frac{1}{2}})}+\left\langle \Lambda\phi_0,\partial_\nu v\right\rangle_{L^2((0,T)\times\partial\Omega)}+\int_0^T\int_\Omega \phi Fdx dt.$$
Using this property, the fact that, since $f_1=f_2$ on $\partial\Omega$, we have $f\in H^1_0(\Omega)$ and the fact that $p\in C^2([0,+\infty);L^2(\Omega))\cap C^1([0,+\infty);H^1(\Omega))\cap C([0,+\infty);H^2(\Omega))$,  we obtain the following identity
$$\begin{aligned}&\left\langle \phi_0,f\right\rangle_{D(L^{-\frac{1}{2}});D(L^{\frac{1}{2}})}-\left\langle\Lambda\phi_0,\partial_\nu p \right\rangle_{L^2((0,T)\times\partial\Omega)}+\left\langle\partial_\nu\phi, p \right\rangle_{\vphantom{H^1((0,T)\times\partial\Omega)}_0 H^1((0,T)\times\partial\Omega)',\vphantom{H^1((0,T)\times\partial\Omega)}_0 H^1((0,T)\times\partial\Omega)}\\
&=\int_0^T\int_\Omega (\partial_t^2p-c_1^{2}\Delta_x p)\phi dxdt=\int_\Omega c_2^{2}(c_1^{-2}-c_2^{-2})\int_0^T\partial_t^2p_1\phi dt dx.\end{aligned}$$
Now using the fact that $\partial_\nu p=-\beta\partial_tp$ on $(0,T)\times\partial\Omega$, we obtain

\begin{equation}\label{tt1}\begin{aligned}&\left\langle \phi_0,f\right\rangle_{D(L^{-\frac{1}{2}}),D(L^{\frac{1}{2}})}+\left\langle\Lambda\phi_0,\beta\partial_t p \right\rangle_{L^2((0,T)\times\partial\Omega)}+\left\langle\partial_\nu\phi, p \right\rangle_{\vphantom{H^1((0,T)\times\partial\Omega)}_0 H^1((0,T)\times\partial\Omega)',\vphantom{H^1((0,T)\times\partial\Omega)}_0 H^1((0,T)\times\partial\Omega)}\\
&=\int_\Omega c_2^{2}(c_1^{-2}-c_2^{-2})\int_0^T\partial_t^2p_1\phi dt dx.\end{aligned}\end{equation}
Let us consider the operator $K$ defined by
$$K\phi_0=\int_0^T\partial_t^2p_1\phi dt.$$
Using the fact that $p_1\in  C^2([0,+\infty);L^2(\Omega))$, $\phi\in C([0,T];L^2(\Omega))$ and applying Lemma \ref{l4} and estimate \eqref{p1bb}, \eqref{t6aa}, we obtain
$$\norm{K\phi_0}_{L^1(\Omega)}\leq T \norm{p_2}_{C^2([0,T];L^2(\Omega))}\norm{\phi}_{C([0,T];L^2(\Omega))}\leq C_*\norm{f_1}_{H^2(\Omega)}\norm{\phi_0}_{D(L^{-\frac{1}{2}})},$$
with a constant $C_*>0$ depending on  $\Omega$ and $\beta$.
Thus $K\in\mathcal B(D(L^{-\frac{1}{2}});L^1(\Omega))$ and we have
$$\norm{K}_{\mathcal B(D(L^{-\frac{1}{2}});L^1(\Omega))}\leq C_*\norm{f_1}_{H^2(\Omega)}.$$
In light of \eqref{L}, can define $K^*\in \mathcal B(L^\infty(\Omega);D(L^{\frac{1}{2}}))=\mathcal B(L^\infty(\Omega);H^{1}_0(\Omega))$  to be the adjoint operator of $K$ as follows
$$\int_\Omega \psi K\phi_0dx=\left\langle \phi_0, K^*\psi\right\rangle_{D(L^{-\frac{1}{2}});D(L^{\frac{1}{2}})},\quad \psi\in L^\infty(\Omega)$$
and, applying \eqref{L}, we recall that 
\begin{equation}\label{tt2}\norm{K^*}_{\mathcal B(L^\infty(\Omega);H^{1}_0(\Omega))}=C'\norm{K^*}_{\mathcal B(L^\infty(\Omega);D(L^{\frac{1}{2}}))}=C'\norm{K}_{\mathcal B(D(L^{-\frac{1}{2}});L^1(\Omega))}\leq C'C_*\norm{f_1}_{H^2(\Omega)},\end{equation}
with $C'>0$ depending only on $\Omega$.
In the same way, we define the map $\mathcal M:D(L^{-\frac{1}{2}})\ni \phi_0\mapsto \partial_\nu\phi\in \vphantom{H^1((0,T)\times\partial\Omega)}_0 H^1((0,T)\times\partial\Omega)'$ where $\phi$ solves in the transposition sense \eqref{eq6}. In view of Lemma \ref{l4} and the estimates \eqref{L} and  \eqref{l4m}, we have $\mathcal M\in \mathcal B(D(L^{-\frac{1}{2}});\vphantom{H^1((0,T)\times\partial\Omega)}_0 H^1((0,T)\times\partial\Omega)')$ and there exists a constant $C_*'>0$ depending only on $\Omega$  such that 
\begin{equation}\label{tt2a}\norm{\mathcal M^*}_{\mathcal B(\vphantom{H^1((0,T)\times\partial\Omega)}_0 H^1((0,T)\times\partial\Omega); H^{1}_0(\Omega))}\leq C'\norm{\mathcal M^*}_{\mathcal B(\vphantom{H^1((0,T)\times\partial\Omega)}_0 H^1((0,T)\times\partial\Omega); D(L^{\frac{1}{2}})}=C'\norm{\mathcal M}_{\mathcal B(D(L^{-\frac{1}{2}});\vphantom{H^1((0,T)\times\partial\Omega)}_0 H^1((0,T)\times\partial\Omega)')}\leq C_*'.\end{equation}
Thus,  \eqref{tt1} can be rewritten as
$$\left\langle \phi_0,K^*[c_2^{2}(c_1^{-2}-c_2^{-2})]-f\right\rangle_{D(L^{-\frac{1}{2}});D(L^{\frac{1}{2}})}=\left\langle\phi_0,\Lambda^*[\beta\partial_t p|_{(0,T)\times\partial\Omega}] +\mathcal M^* [p|_{(0,T)\times\partial\Omega}]  \right\rangle_{D(L^{-\frac{1}{2}});D(L^{\frac{1}{2}})}.$$
Since here $\phi_0\in D(L^{-\frac{1}{2}})$ is arbitrary chosen, we obtain the identity \eqref{t6a} by choosing $\mathcal G= K^*$, $\mathcal I=\Lambda^*$ and $\mathcal J=\mathcal M^*$. Moreover, we deduce \eqref{t6b} and \eqref{t6bbb} by applying the estimates \eqref{l4m} and \eqref{tt2}-\eqref{tt2a}. This completes the proof of the theorem.\qed

\section{Proof of Theorem \ref{t1}}

We fix $a\in(\frac{3}{4},1)$ to be determined and  denote  $f_j = f(c_j, \cdot)$.  Applying Theorem \ref{t6}, we obtain the following representation

\bel{t1c}\mathcal G[c_2^{2}(c_1^{-2}-c_2^{-2})]-(f_1-f_2)= \mathcal I[\beta\partial_t p|_{(0,T)\times\partial\Omega}] +\mathcal J [p|_{(0,T)\times\partial\Omega}],\ee
with $p=p_1-p_2$ solving \eqref{eq3}. In view of \eqref{ass1} and \eqref{t6b}, there exists a constant $C>0$ depending on $\Omega$ and $M$ such that
\bel{t1d}\norm{\mathcal G[c_2^{2}(c_1^{-2}-c_2^{-2})]}_{H^1(\Omega)}\leq C\norm{c_2^{2}(c_1^{-2}-c_2^{-2})}_{L^\infty(\Omega)}=C\norm{c_1^{-2}-c_2^{-2}}_{L^\infty(\Omega)}.\ee
On the other hand, we have

$$c_1^{-2}-c_2^{-2}= (a^{-2}-1)(\mathds{1}_{\omega_2}-\mathds{1}_{\omega_1} ). $$

Recalling that, for $j=1,2$, $\frac{1}{2}\leq a\leq c_j\leq 1$, we obtain
$$ \norm{c_1^{-2}-c_2^{-2}}_{L^\infty(\Omega)}\leq 6(1-a)\norm{\mathds{1}_{\omega_2}-\mathds{1}_{\omega_1}}_{L^\infty(\Omega)}.$$

Combining this with \eqref{t1d}, we obtain
$$\norm{\mathcal G[c_2^{2}(c_1^{-2}-c_2^{-2})]}_{H^1(\Omega)}\leq 6C(1-a)\norm{\mathds{1}_{\omega_2}-\mathds{1}_{\omega_1}}_{L^\infty(\Omega)}.$$
Fixing
$$a_0=\max\left( \frac{3}{4}, 1-\frac{d}{6C}\right),$$
with $d>0$ the constant appearing in \eqref{ass2}, and applying \eqref{ass2}, for any $a\in(a_0,1]$, we find
$$\begin{aligned}\norm{\mathcal G[c_2^{2}(c_1^{-2}-c_2^{-2})]-(f_1-f_2)}_{H^1(\Omega)}&\geq \norm{f_1-f_2}_{H^1(\Omega)}-\norm{\mathcal G[c_2^{2}(c_1^{-2}-c_2^{-2})]}_{H^1(\Omega)}\\
&\geq (d-6C(1-a))\norm{\mathds{1}_{\omega_2}-\mathds{1}_{\omega_1}}_{L^\infty(\Omega)},\end{aligned}$$
where $(d-6C(1-a))>(d-6C(1-a_0))\geq0.$
Combining this with \eqref{t6bbb} and \eqref{t1c}, for any $a\in(a_0,1]$, we obtain
$$\norm{\mathds{1}_{\omega_2}-\mathds{1}_{\omega_1}}_{L^\infty(\Omega)}\leq\left(\frac{C\norm{\beta}_{L^\infty(\partial\Omega)}+C}{(d-6C(1-a))}\right)\norm{p}_{H^1((0,T)\times\partial\Omega)}$$
which clearly implies \eqref{maineq1}.

In order to complete the proof of the theorem, we only need to show \eqref{maineq2}. From now on, we denote by $C>0$ a constant depending on $\Omega$, $\beta$, $a$ and $M$ that may change from line to line. Recalling that $f_1=f_2$ and $\partial_\nu f_1= \partial_\nu f_2$ on $\partial\Omega$, we have
$$p(0,x)=f_1(x)-f_2(x)=0,\quad \partial_tp(0,x)=g_1(x)-g_2(x)=\beta^{-1}(x)(\partial_\nu f_2(x)-\partial_\nu f_1(x))=0,\quad x\in\partial\Omega.$$
Therefore, in view of \cite[Chapter I, Theorem 2.3]{Li2}, we can find $G\in H^2((0,T)\times\Omega)$ such that
\bel{t1f} G|_{(0,T)\times\partial\Omega}=p|_{(0,T)\times\partial\Omega},\quad G(0,\cdot)=\partial_tG(0,\cdot)\equiv0,\ee
\bel{t1g}\norm{G}_{H^2((0,T)\times\Omega)}\leq C_1\left(\|  p_1 -   p_2\|_{H^{\frac{3}{2}}((0, T)\times\partial \Omega  )}+\|  t^{-\frac{1}{2}}(\partial_t p_1 -   \partial_tp_2)\|_{L^2((0, T)\times\partial \Omega  )}\right),\ee
with $C_1>0$ depending only on $\Omega$ and $T$. Then, we can decompose $p$ into two terms $p=v+G$ such that $v$ solves the problem
\begin{equation}\label{eq6}\left\{\begin{array}{ll}\partial_t^2v-c_2^{2}\Delta_x v=c_2^{2}(c_1^{-2}-c_2^{-2})\partial_t^2p_1-(\partial_t^2G-c_2^{2}\Delta_x G),\quad &\textrm{in}\ (0,T)\times\Omega,\\  v(0,\cdot)=f,\quad \partial_tv(0,\cdot)=g,\quad &\textrm{in}\ \Omega,\\ v=0,\quad &\textrm{on}\ (0,T)\times\partial\Omega,\end{array}\right.\end{equation}
Applying the observability inequality of Proposition \ref{p3}, we obtain
$$\begin{aligned}\norm{f}_{H^1(\Omega))}^2&\leq C\left(\norm{\partial_\nu v}_{L^2((0,T)\times\partial\Omega)}^2+\norm{c_2^{2}(c_1^{-2}-c_2^{-2})\partial_t^2p_1-(\partial_t^2G-c_2^{2}\Delta_x G)}^2_{L^2((0,T)\times\Omega)}\right)\\
&\leq C\left(\norm{\partial_\nu v}_{L^2((0,T)\times\partial\Omega)}^2+2\norm{c_2^{2}(c_1^{-2}-c_2^{-2})\partial_t^2p_1}^2_{L^2((0,T)\times\Omega)}+2\norm{(\partial_t^2G-c_2^{2}\Delta_x G)}^2_{L^2((0,T)\times\Omega)}\right)\\
&\leq C\left(\norm{\partial_\nu v}_{L^2((0,T)\times\partial\Omega)}^2+\norm{\mathds{1}_{\omega_2}-\mathds{1}_{\omega_1}}_{L^\infty(\Omega)}^2\norm{p_1}_{C^2([0,T];L^2(\Omega))}^2+\norm{G}_{H^2(0,T)\times\Omega)}^2\right).\end{aligned}$$
Combining this with \eqref{ass1}, \eqref{maineq1}, Proposition \ref{p1} and \eqref{t1g}, we obtain
\bel{t1h}\norm{f}_{H^1(\Omega))}\leq C\left(\norm{\partial_\nu v}_{L^2((0,T)\times\partial\Omega)}+\|  p_1 -   p_2\|_{H^{\frac{3}{2}}((0, T)\times\partial \Omega  )}+\|  t^{-\frac{1}{2}}(\partial_t p_1 -   \partial_tp_2)\|_{L^2((0, T)\times\partial \Omega  )}\right)\ee
Meanwhile, we get
$$\partial_\nu v(t,x)=\partial_\nu p(t,x)-\partial_\nu G(t,x)=-\beta(x)\partial_tp-\partial_\nu G(t,x),\quad (t,x)\in(0,T)\times\partial\Omega$$
and it follows
$$\begin{aligned}\norm{\partial_\nu v}_{L^2((0,T)\times\partial\Omega)}&\leq \norm{\beta}_{L^\infty(\partial\Omega)}\norm{\partial_t p}_{L^2((0,T)\times\partial\Omega)}+\norm{\partial\nu G}_{L^2((0,T)\times\partial\Omega)}\\
&\leq C(\norm{p}_{H^1((0,T)\times\partial\Omega)}+\norm{ G}_{H^2((0,T)\times\Omega)})\\
&\leq C\left(\|  p_1 -   p_2\|_{H^{\frac{3}{2}}((0, T)\times\partial \Omega  )}+\|  t^{-\frac{1}{2}}(\partial_t p_1 -   \partial_tp_2)\|_{L^2((0, T)\times\partial \Omega  )}\right).\end{aligned}$$
Combining this last estimate with \eqref{t1h}, we obtain \eqref{maineq2}. This completes the proof of the theorem.

   
%

\end{document}